\documentclass[a4paper,latexcad,twoside,11pt]{article}
\usepackage{amssymb,amsmath} 
\usepackage{graphicx}
\usepackage{epstopdf}

\usepackage{ntheorem}
\newtheorem{theorem}{Theorem}
\newtheorem{lemma}{Lemma}

\newtheorem{proposition}{Proposition}

\newtheorem{definition}{Definition}

\newenvironment{proof}{{
		\bf Proof\quad}}{\hfill $\hfill\square$\par}

\newtheorem{problem}{Problem}
\def\dfrac{\displaystyle\frac}

\usepackage{color}

\newcommand \red[1]
{\textcolor{red}{#1}}

\usepackage{float} 

\newcommand \equ[2]
{\begin{equation}\label{#1}
		#2
	\end{equation}
}

\newcommand \eqn[2]
{\begin{eqnarray}\label{#1}
		#2
	\end{eqnarray}
}

\newcommand \Bracket[1]
{
	\left	( {#1} \right )
}

\usepackage{stmaryrd}

\newcommand \brk[1]
{
	\llbracket #1\rrbracket
}

\def \hyh {{\mathcal H}}
\def \setf {{\mathcal F}}

\def \dom {\mbox{dom}} 

\def \N {{\mathbb{N}}}

\begin{document}
\title{ DP color functions of hypergraphs }
\author{Ruiyi Cui$^1$, 
Liangxia Wan$^1$\thanks{
		Emails: lxwan@bjtu.edu.cn. },  
Fengming Dong$^2$\thanks{
	Email: fengming.dong@nie.edu.sg}
 \\
 \small $^1$School of Mathematics and Statistics,
Beijing Jiaotong University, Beijing $100044$, China
\\
\small $^2$National Institute of Education, Nanyang Technological University, Singapore
}

\date{}
\maketitle

\abstract{
	In this article, we introduce the DP color function 
	$P_{DP}(\hyh,k)$ 
	of a hypergraph $\hyh$ for $k\in \mathbb N$, based on the DP coloring
	introduced by Bernshteyn and Kostochka, which is the minimum value of $P_{DP}(\mathcal H,\mathcal F)$ 
	all $k$-fold covers $\mathcal F$ of $\mathcal H$. It	
	is an extension 
	of its chromatic polynomial 
	$P(\hyh,k)$ with the property
	that $P_{DP}(\hyh,k)\le P(\hyh,k)$
	for all positive integers $k$.
We obtain an 
	upper bound for $P_{DP}(\hyh,k)$
	when $\hyh$ is a connected 
	$r$-uniform hypergraph for $r\ge 2$, 
	and the upper bound is attained 
	if and only if $\hyh$ is a 
$r$-uniform hypertree. 
	We also show that
$P_{DP}(\hyh,k)= P(\hyh,k)$ holds 
when $\hyh$ is a 
$r$-uniform hypertree or a unicycle linear $r$-uniform hypergraph with odd cycle for $r\ge 3$. These conclusions coincide with  
the known results of graphs.

}
\vskip 2mm
\noindent{\small{\bf Mathematics Subject Classiﬁcation} 05C15, 05C30, 05C65}
\vskip 2mm
\noindent{\bf keywords} Chromatic polynomial$\cdot$ Hypergraph coloring$\cdot$ DP-coloring$\cdot$ DP color function

\section{Introduction}

\noindent

In this paper, each hypergraph is nonempty and simple unless otherwise noted. 
Let $\mathbb{N}$ denote the set of positive integers, and  let 
$\brk{k}$ denote the set $\{1,2,\cdots,k\}$ for $k\in \mathbb{N}$. 
For a hypergraph $\mathcal{H}$,
let $V(\hyh)$ and $E(\hyh)$ denote the 
vertex set and edge set of $\hyh$,
respectively.
For any $k\in \N$, a proper $k$-coloring 
of $\hyh$ is a mapping $f:V(\hyh)\mapsto \brk{k}$
such that for any edge $e\in E(\hyh)$,
there exist a pair of vertices $u,v\in e$
such that $f(u)\ne f(v)$.
This definition clearly includes the
special  case 
that $\hyh$ is a graph. 
The concept of vertex-coloring was
extended to list-coloring 
 independently by Vizing \cite{Vi76} and Erd\H{o}s, Rubin, and Taylor \cite{ERT} in the 1970s. Dvo\v{r}\'{a}k and Postle
~\cite{DP18}
 further  extended list-coloring to DP-coloring (or correspondence coloring) in order to prove that every planar graph without cycles of lengths $4$ to $8$ is $3$-choosable in 2015.

For any graph $G$,  its 
{\it chromatic polynomial}, 
denoted by $P(G,k)$,
is the function counting 
distinct $k$-colorings of $G$. 
This function, indeed a polynomial
in $k$, 
 was introduced by Birkhoff in 1912 with the hope of proving
the Four Color Conjecture \cite{Bi12}. 
This function for a graph is naturally 
extended to that for a hypergraph.
Readers can refer to \cite{DL22, To98, To09, Tr14, WQY, ZD17, ZD20}
for recent works on chromatic 
polynomial.
For any $k\in \N$, 
the list color function 
$P_l(G,k)$ of $G$ is the minimum value 
of $P(G,L)$ over all $k$-assignments
$L$, where  $P(G,L)$ 
is the number of $L$-coloring of $G$.
A survey paper on $P_l(G,k)$ was 
provided 
by Thomassen \cite{Th09}.

\def \seth {{\cal C}}

DP-coloring is a generalization of list coloring,
and a formal definition is given below.
For a graph $G$, a {\it cover} of $G$ is an ordered pair 
$\seth = (L, H)$, 
where $H$ is a graph and 
$L$ is a mapping from $V(G)$ to the power set of $V(H)$
satisfying the four conditions below:
\begin{enumerate}
	\renewcommand \theenumi {(\roman{enumi})}
	
	\item  the set $\{L(u) : u \in V (G)\}$ is a partition of 
$V(H)$;

\item for every $u\in V (G)$, 
$H[L(u)]$ is a complete graph;

\item  if $u$ and $v$ are non-adjacent vertices in $G$,
then $E_H (L(u), L(v))=\emptyset$; and 

\item  for each edge $uv\in E(G)$,  
$E_H (L(u), L(v))$ is a matching.
\end{enumerate}

An {\it $\seth$-coloring} of $G$ is 
defined to be 
an independent set $I$ of $H$ with $|I|=|V(G)|$. 
Clearly,  
$|I\cap L(u)| = 1$ holds for each $u \in V (G)$
and each $\seth$-coloring $I$ of $G$. 
Let $P_{DP}(G,\seth)$ be the number of
$\seth$-colorings of $G$.
A cover $\seth = (L, H)$  of $G$ is called 
an {\it $m$-fold} cover if 
$|L(u)| = m$ for each $u\in V (G)$.
For each $m\in \N$,
let $P_{DP}(G,m)$ be the minimum value 
of $P_{DP}(G,\seth)$ over 
all $m$-fold covers $\seth$ of $G$,
and $P_{DP}(G,m)$ is called 
the {\it DP color function} of $G$,
introduced by 
Kaul and Mudrock~\cite{KM21}. 

In ~\cite{KM21} and  \cite{MT21},
Kaul and Mudrock
showed that for any $g\ge 3$,
there exists 
a graph $G$ with girth $g$ such that $P_{DP}(G, k) < P(G, k)$ holds
for sufficiently large $k$,
and for any graph $G$, 
$P_{DP}(G\vee K_1, k) 
= P(G\vee K_1, k)$ holds for sufficiently large $k$,
where $H_1\vee H_2$
is the graph with vertex set 
$V(H_1)\cup V(H_2)$ and edge set 
$E(H_1)\cup E(H_2)\cup \{xy: x\in V(H_1), y\in V(H_2)\}$.
 For each edge $e$ in $G$, let $l(e) = \infty$ if $e$ is a bridge of $G$, and let $l(e)$ be the length of a shortest cycle in $G$ containing
$e$ otherwise. Dong and Yang \cite{DY22} 
showed that if $l(e)$ is even for some edge $e$ in $G$, then $P_{DP}(G,k) < P(G,k)$ holds for sufficiently large $k$.
They also  proved that $P_{DP}(G,k)=P(G,k)$ 
holds for sufficiently large $k$
if $G$ contains a spanning tree $T$ such that 
for each $e\in E(G)\setminus E(T)$, $l(e)$ is odd and $e$ is contained in a cycle $C$ of length $l(e)$ with the property that $l(e') < l(e)$ for each $e'\in E(C)\setminus (E(T)\cup {e})$. 
Some other recent works 
on DP-color functions 
are referred to \cite{BHKM, DK24, 
	HKLMSS, KMSS, LY24}.

In this article, we introduce the DP color function of a hypergraph.
We obtain an upper bound 
for the DP color function of a 
$r$-uniform hypergraph $\hyh$,
 extending  the 
analogous result of the DP function of a  graph due to Kaul and Mudrock \cite{KM21}. 
We show that this upper bound is attained if and only if  $\cal H$ is a $r$-uniform hypertree for $r\ge 2$.
Corollary $4$ in \cite{KM21} is the 
special case $r=2$ in our result. 
We also find a  formula for 
$P_{DP}(\hyh,k)$ when $\hyh$ is 
a unicycle linear $r$-uniform hypergraph for $r\ge 3$. 
This result extends the Theorem $11$ in \cite{KM21}.

This paper is organized as follows.
In Section $2$, we mainly provide 
some known results 
on the expression of chromatic polynomials of some special uniform hypergraphs, and the definition of 
$k$-fold covers of a hypergraph.   
 In Section 3, we introduce the DP color function of a hypergraph and show several inequalities of the DP color functions of a hypergraph. 
 In Section 4, 
 we prove that the DP color function of any connected $r$-uniform hypergraph $\mathcal H$ equals to the upper bound if and only if $\cal H$ is a $r$-uniform hypertree for $r\ge 2$. 
 Moreover, the explicit formula for the DP color function of any unicycle linear $r$-uniform hypergraph is provided for $r\ge 3$.  Several problems are proposed in Section $5$.

\section{Preliminaries}

In this section, we state some concepts \cite{Be73} and review some related results about the chromatic polynomial of a hypergraph \cite{BK19, BŁ07}.

\subsection{Chromatic polynomials of some hypergraphs}

\vskip 3mm

A {\it hypergraph} $\hyh = (X,E)$ consists of a finite non-empty set
$X$ and a subset $E$ of $2^X$ 
(i.e., the power set of $X$) with 
$|e|\ge 2$ for each $e\in E$.
 A hypergraph $\hyh$ is 
 said to be {\it simple} if 
 $e_1\not\subseteq e_2$ for any two 
 distinct edges $e_1$ and $e_2$ in $\hyh$, 
is {\it $r$-uniform} if $|e| = r$ for each 
$e\in E$, and  is {\it linear} if 
$|e_1\cap e_2|\le 1$ for any two 
distinct edges $e_1$ and $e_2$ in $\hyh$. 
The number of edges in $\hyh$ 
containing a vertex $v$ is its degree $d_{\mathcal H}(v)$.
A {\it cycle} $C$ of length $l$ for $l\ge 3$ in $\mathcal H$ is a subhypergraph consisting of  
$l$ distinct edges
$e_1, e_2,\cdots, e_l$ such that the
vertex $v_i\in e_i\cap e_{i+1}$ exits 
for each $i\in \brk{l-1}$, 
and $v_l\in e_1\cap e_{l}$,
and 
cycle $C$ is
 said to be {\it elementary} if 
 $d_{C}(v_i)=2$ for each 
 $i\in \brk{l}$ and 
 $d_{C}(v)=1$ for any other
 $v\in \cup_{i\in \brk{l}}e_i$.
 
 For any edge $e$ in $\hyh=(X,E)$, 
 $\hyh-e$ denotes the hypergraph $(X,E\setminus\{e\})$.
 A {\it path} $P$ in $\mathcal H$ is a subhypergraph consisting of 
 edges
 $e_1, e_2,\cdots, e_t$ in $\hyh$ such that
 $e_i\cap e_{i+1}\ne \emptyset$ 
 for each $i\in \brk{t-1}$.
 $\hyh$ is said to be {\it connected}
 if for any two vertices $u$ and $v$ in $\hyh$, there exists a path $P$ in $\hyh$ 
 such that both $u$ and $v$ are in $P$.
A {\it hypertree} is a connected linear hypergraph without cycles. 
A {\it unicyclic hypergraph} is a connected hypergraph
containing exactly one cycle.

Dohmen calculated the chromatic polynomials of hypertrees as follows \cite{Do93}.
\begin{lemma}\label{lem:2-1}
	If $\mathcal T_{m}^{r}$ is a $r$-uniform hypertree with $m$ edges, where $r\geq 2$ and $m\geq 0$, then $$P(\mathcal T_{m}^{r},k)=k(k^{r-1}-1)^{m}.$$
\end{lemma}

Borowiecki and Łazuka provided the chromatic polynomials of $r$-uniform unicyclic hypergraphs with $m$ edges as follows \cite{BŁ07}.
\begin{lemma}\label{lem:2-3}
	Let $\mathcal H$ be a linear hypergraph. $\mathcal H$ is a $r$-uniform unicyclic hypergraph with $m + p$ edges and a cycle of length $p$ if and only if
	$$P(\mathcal H,k)=(k^{r-1}-1)^{m+p}+(-1)^{p}(k-1)(k^{r-1}-1)^{m}$$
	where $r\geq 3,m\geq 0$ and $p\geq 3$.
\end{lemma}

\subsection{DP-colorings of a hypergraph}

The concepts of $k$-fold covers and DP-colorings of hypergraphs were first introduced by Bernshteyn 
and Alexandr Kostochka~\cite{BK19}.

Each map could be identified 
with its graph which is the set $\{(x,y)|f(x)=y\}$. Let $\varphi:X\rightharpoonup Y$ indicate 
a partial map from $X$ to $Y$  
defined on a subset of $X$ 
with images in $Y$. 
The set of all partial maps from $X$ to $Y$ is denoted by $[X\rightharpoonup Y]$.

Consider a family 
$\mathcal F\subseteq 
[X\rightharpoonup \brk{k}]$ of partial maps, where $k\in \N$.
The domain of $\mathcal F$, denoted by $\dom(\mathcal F)$,  is 
$$
{\rm dom}(\mathcal F) :=\{{\rm dom}(\varphi):\varphi\in \mathcal F\}.
$$
The set ${\rm dom}(\mathcal F)$ 
can be considered as a hypergraph on $X$. 
Each $\varphi \in {\mathcal F}$ is a subset of the Cartesian product $X\times \brk{k}$, and 
$\mathcal F$ can be regarded as the  hypergraph on $X\times \brk{k}$
whose edge set is the set of 
$\{(x,\varphi(x)): x\in \dom(\varphi)\}$
for all $\varphi\in \setf$.
A function $f:X\rightarrow \brk{k}$ {\it avoids} a partial map $\varphi:X\rightharpoonup \brk{k}$ if $\varphi \nsubseteq f$. Given a family $\mathcal F\subseteq[X\rightharpoonup \brk{k}]$ of partial maps, a {\it $(k,\mathcal F)$-coloring} (or simply an {\it $\mathcal F$-coloring}) is a function $f:X\rightarrow \brk{k}$ that avoids all $\varphi \in \mathcal F$.

\begin{definition}
	\label{k-fold}
	{\rm Let $\hyh$ be a hypergraph 
	and $k\in \N$. 
A $k$-fold cover of $\mathcal H$ 
is a family of partial maps $\mathcal F\subseteq [X\rightharpoonup \brk{k}]$ such that
${\rm dom}(\mathcal F
)\subseteq V(\hyh)$ and
if ${\rm dom}(\varphi)={\rm dom}(\psi)$ for distinct $\varphi,\psi\in \mathcal F$, then $\varphi \cap \psi=\emptyset$.}
\end{definition}

 Let $\mathcal F_e=\{\varphi\in {\cal F}: {\rm dom}(\varphi)=e\}$ and let $\mathcal F\setminus\mathcal F_e=\{\varphi\in {\cal F}: {\rm dom}(\varphi)\neq e\}$.
 If $\setf$ is a $k$-fold cover of $\hyh$,
 then, ${\mathcal F}\setminus\setf_e$ is a $k$-fold cover of ${\cal H}-e$.
If $\mathcal F$ is a $k$-fold cover of a hypergraph $\mathcal H$, then,  
for each edge $e$ in $\hyh$,
 $$
 |\setf_e|=|\{\varphi\in {\cal F}: {\rm dom}(\varphi)=e\}|
 \leqslant k. $$  
 \begin{definition}{\rm
 	A hypergraph $\mathcal H$ on a set $X$ is  $k$-DP-colorable if each $k$-fold cover $\mathcal F$ of $\mathcal H$ admits a $(k,\mathcal F)$-coloring $f:X\rightarrow \brk{k}$.}
\end{definition}

\section{Some properties of the DP color function }

In this section we define a natural cover and the DP color function of a hypergraph. An upper bound is given for the DP color function of any $r$-uniform hypergraph for $r\ge 2$. Several inequalities of the DP color function for any hypergraph $\mathcal H$ are provided by applying a bijection between its $\mathcal F$-colorings and  proper colorings of ${\cal H}-e$ for $e\in E(\cal H)$.
\vskip 2mm

\begin{definition} {\rm
Let $\mathcal H$ be a hypergraph with 
 $E(\hyh)=\{e_j: j\in \brk{m}\}$.
For any $k\in  \mathbb{N}$,  define 
$$
\iota_{\mathcal H}(k)
=\left \{\varphi^{(j)}_i:  i\in \brk{k},  j\in \brk{m}\right \},
$$ 
where $\varphi^{(j)}_i$ is the mapping: $e_j\mapsto \{i\}$.
$\iota_{\mathcal H}(k)$ is 
called a natural $k$-cover of $\hyh$. }
\end{definition}

If $\mathcal I$ is the
set of $\iota_{\mathcal H}(k)$-colorings of $\mathcal H$ and $\mathcal J$ is the set of proper $k$-colorings of $\mathcal H$, then the function
$f: \mathcal J\rightarrow \mathcal I$ given by
$$ f(c) = \{(v,c(v)): v\in V
(\mathcal H)\}$$
is a bijection between the $\iota_{\mathcal H}(k)$-colorings  of $\mathcal H$ and its proper $k$-colorings.

\begin{definition}{\rm
Suppose that $\mathcal F$ is a $k$-fold cover of a hypergraph $\mathcal H$ for $k\in \mathbb{N}$. Let $P_{DP}(\mathcal H,\mathcal F)$ be the number of $\mathcal F$-colorings of $\mathcal H$. The DP color
	function denoted by $P_{DP}(\mathcal H,k)$ is the minimum value of $P_{DP}(\mathcal H,\mathcal F)$
all $k$-fold covers $\mathcal F$ of $\mathcal H$.}
\end{definition} 

It is obvious that
for any hypergraph $\mathcal H$ and $k\in \mathbb{N}$
\begin{equation}
	P_{DP}(\mathcal H,k)\le P(\mathcal H,k).
	\label{eq:3-1}
\end{equation}
If $\mathcal H$ is a hypergraph with components: $\mathcal H_1, \mathcal H_2,\cdots,\mathcal H_t$, then
$$
P_{DP}(\mathcal H,k)=\prod_{i=1}^t P_{DP}(\mathcal H_i,k).$$
So, it is enough to consider the connected hypergraphs 
for the study of 
DP color functions of hypergraphs.

\begin{proposition} Suppose 
	$\mathcal H=(X,E)$ is a $r$-uniform hypergraph with $X=\{x_1,x_2,\cdots,x_n\}$ for $r\ge 2$. Then, for each $k\in \mathbb{N}$, $$P_{DP}(H,k)\leq \dfrac{k^{n}(k^{r-1}-1)^{|E(H)|}}{k^{(r-1)|E(H)|}}.$$
	\end{proposition}
\begin{proof} 
The result holds obviously if $\mathcal{H}$ is edgeless. Suppose that $E(\mathcal{H})
=\{e_1,e_2,\cdots,e_m\}$,
where $m\geq 1$.
Notice that $|[X \rightarrow \brk{k}]|=k^{n}$.
We index the elements 
in $[X \rightarrow \brk{k}]$
as $f_{1},f_{2},\cdots, f_{k^{n}}$.

A $k$-fold cover $\mathcal F$ can be partial randomly formed by choosing a color for each $u\in V(\cal H)$ such that $\varphi^{(j)}_p(u)\ne \varphi^{(j)}_q(u)$ with $p\ne q$ for any $p,q\in \brk{k}$, each $j\in \brk{m}$. Denote the corresponding family partial maps by 
$$
\mathcal{F}
=\{\varphi^{(j)}_i: i\in \brk{k}, 
j\in \brk{m}\}, 
$$ 
where 
$\varphi^{(j)}_i: e_j\rightarrow \brk{k}$.
Let $A_{s}$ be the event that 
$f_{s}$ is a $\mathcal F$-coloring of $\mathcal H$ for $s \in [k^{n}]$.
For any $f\in [X\rightarrow \brk{k}]$
 and $i\in \brk{k}$ and $j\in \brk{m}$, 
 the probability that $f$ avoids $\varphi^{(j)}_i$ in the $k$-fold cover 
 is
$$
1-\Bracket{\dfrac{1}{k}}^{r-1}.
$$
Thus, 
$$
P(A_{s})
=\Bracket{1-\Bracket{\dfrac{1}{k}}^{r-1}}
^{m}.
$$
Let $Y_{s}$ be the random variable that is one if $A_{s}$ occurs and zero otherwise.
Let 
$$Y=\sum\limits_{s=1}^{k^{n}}Y_{s}.$$
Obviously, $Y$ is a random variable and $Y$ is equal to the number of $\mathcal F$-colorings of $\mathcal H$. Because of the linearity of expectation, 
$$E(Y)
=E\Bracket{\sum_{s=1}^{k^{n}}Y_{s}}
=\sum_{s=1}^{k^{n}}
\Bracket{1-\Bracket{\frac{1}{k}}^{r-1}}^{m}
=\frac{(k^{r-1}-1)^{m}\cdot k^{n}}{k^{(r-1)m}}.
$$
The result holds. 
\end{proof}

\vskip 2mm
The following result is obvious for a proper coloring of a hypergraph $\mathcal H$.

\begin{lemma}\label{lem:3-1}
	 Suppose that $\mathcal H$ is a hypergraph and 
	 $e=\{v_i: i\in \brk{r}\}\in E(\hyh)$, 
	 where $r\ge 2$.
	 Assume that $v_i$ can be colored arbitrarily for each $i\in \brk{r}$ in $\cal H$ $-e$. 
	 For any $i_1,i_2,\ldots, i_r$, 
	 where 
	 $i_{j}\in \brk{k}$ for each $j\in \brk{r}$, 
	 let $S_{k}^{(i_1,i_2,\cdots,i_{r})}$ be the set of proper $k$-colorings of $\mathcal H-e$ that color $v_{j}$ with $i_{j}$ for each $j\in \brk{r}$.
Then
\begin{enumerate}
	\item[(1)] 
	$t_1:=|S_{k}^{(i,\cdots,i)}|$ is a constant for all $i\in \brk{k}$, and 
	$$
	kt_1=P(\mathcal H-e,k)
	-P(\mathcal H,k);
	$$
\item[(2)] for any given
$i_1,i_2,\cdots,i_{r}\in \brk{k}$, if 
$i_s\ne i_t$ for a pair of $s,t$ with 
$1\le s<t\le r$,   then 
$t_2:=|S_{k}^{(i_1,i_2,\cdots,i_{r})}|$
is a constant, and 
$$
(k^{r}-k)t_2=k(k^{r-1}-1)t_2=P(\mathcal H,k).
$$
\end{enumerate} 
 \end{lemma}

\begin{lemma}\label{lem:3-2}
	Let $\mathcal F$ be a $k$-fold cover of a hypergraph $\mathcal H$ 
	and $k\geq 2$. 
	Suppose that 
	$e=\{v_i: i\in \brk{r}\}\in E(\mathcal H)$, where $r\ge 2$.
Then, 
$\mathcal F'=\mathcal F-\mathcal F_{e}$ is a $k$-fold cover of $\mathcal H-e$,
where 	
$\mathcal F_{e}=
\{ \varphi\in \mathcal F:
\dom (\varphi) =e\}$.
If there is a 
bijection between $\mathcal F'$-colorings of $\mathcal H-e$ and proper $k$-colorings of $\mathcal H-e$, then	
there exists a $k$-fold cover $\mathcal F^*$ of $\hyh$ such that 
\eqn{eq1}
{
	P_{DP}(\mathcal H,\mathcal F)
	&\geqslant  &
	P_{DP}(\mathcal H,\mathcal F^*) 
	\nonumber \\
	&=&
	\min \left \{
	P(\hyh, k), 
	P(\mathcal H-e,k)	-
	\frac{P(\mathcal H,k)}{k^{r-1}-1}
	\right \}.
}

\end{lemma}
\begin{proof}
	Because there is a 
	bijection between $\mathcal F'$-colorings of $\mathcal H-e$ and its $k$-proper colorings, there are exactly  $P(\mathcal H-e,k)$ $\mathcal F'$-colorings of $\mathcal H-e$. 
	If $v_1, v_2, \cdots, v_r$ all have the same color $i\in \brk{k}$, 
	then there are precisely $(P(\mathcal H-e,k)-P(\mathcal H,k))/k$ $\mathcal F$-colorings of $\mathcal H-e$ which
	contain $(v_{j},{i})$ for all $j\in \brk{r}$ by Lemma \ref{lem:3-1}. 
	Similarly, if each $v_j$ is assigned 
	a color in $\brk{k}$
	and there exists 
	$1\le s<t\le r$ such that 
	$v_s$ and $v_t$ are colored differently, 
	then there are precisely $\dfrac{P(\mathcal H,k)}{k(k^{r-1}-1)}$ $\mathcal H'$-colorings of $\mathcal H-e$ which contain $(v_{j},i_{j})$ for all $j\in \brk{r}$.
	
	Since $|\mathcal F_{e}|\leqslant k$, 
$$P_{DP}(\mathcal H,\mathcal F)\geqslant P(\mathcal H-e,k)-{\rm max}\{P(\mathcal H-e,k)-P(\mathcal  H,k),\frac{P(\mathcal H,k)}{k^{r-1}-1}\}.$$
	A $k$-fold cover $\mathcal F^{*}$ of $\mathcal H$ is constructed as follows.
	If $$P(\mathcal H-e,k)-P(\mathcal H,k)\geqslant \frac{P(\mathcal H,k)}{k^{r-1}-1},
	$$ 
	then, for each $i\in \brk{k}$,
	add $\varphi_i$ with $\varphi(v_j)=i$ for each $j\in \brk{r}$ based on $\mathcal F'$.
	If
	$$P(\mathcal H-e,k)-P(\mathcal H,k)<\frac{P(\mathcal H,k)}{k^{r-1}-1},
	$$ 
	then, $\varphi_{i}:e\rightarrow \brk{k}$ for each $i\in \brk{k}$ such that $\varphi_{i}(v_{1})=i+1$ and $\varphi_{i}(v_{j})=i$ for $2\le j\le r$. 
	It is easy to verify that $\mathcal F^{*}$ has the desired property. Therefore the result holds.
\end{proof}

\begin{proposition} \label{prop1}
Suppose that $\mathcal H$ is a hypergraph and
 $e=\{v_i: i\in \brk{r}\}
 \in E(\mathcal H)$,
 where $r\ge 2$. 
 If 
 $P(\mathcal H-e,k)<\dfrac{k^{r-1}}{k^{r-1}-1}P(\mathcal H,k)$,
then 
$$
P_{DP}(\mathcal H,k)<P(\mathcal H,k).
$$
\end{proposition}
\begin{proof}
Let ${\cal H}'={\mathcal H}-e$.
There
is a bijection between the $ \iota_{\mathcal H'}$-colorings  of $\mathcal H'$ and its proper $k$-colorings.
So, Lemma \ref{lem:3-2} implies that there exist an $k$-fold cover $\mathcal F$ of $\mathcal H$ such that
$$
P_{DP}(\mathcal H, \mathcal F)=\min
\left \{P(\mathcal H, k), P(\mathcal H-e,k)-
\frac{P(\mathcal H,k)}{k^{r-1}-1}\right \}.
$$
Since 
$$
P(\mathcal H-e,k)<\frac{k^{r-1}}{k^{r-1}-1}P(\mathcal H,k)
$$
it follows that 
$$P(\mathcal H-e,k)-\dfrac{P(\mathcal H,k)}{k^{r-1}-1}<P(\mathcal H,k).$$
Then
$$P_{DP}(\mathcal H,k)\leqslant P_{DP}(\mathcal H,\mathcal F)=P(\mathcal H-e,k)-\dfrac{P(\mathcal H,k)}{k^{r-1}-1}<P(\mathcal H,k).$$ 
Thus the result is deduced.
\end{proof}

\section{DP color functions of hypertrees and unicycle hypergraphs}

In this section we show that the DP color function of any connected $r$-uniform hypergraph $\mathcal H$ equals to the upper bound if and only if $\cal H$ is a $r$-uniform hypertree for $r\ge 2$. We provide the explicit formula for the DP color function of any unicycle linear $r$-uniform hypergraph for $r\ge 3$.

	\begin{lemma}\label{lem:3-3}
		Suppose that $\mathcal H$ is a hypergraph with 
	$m$	edges 
		and has no cycles for $m\ge 0$. 
		Suppose that $\mathcal F=\{\varphi_i^{(j)}: i\in \brk{k}, j\in \brk{m}\}$ is a $k$-fold cover
		of $\mathcal H$ for $k\ge 2$. 
		Then there is a 
		bijection between $\mathcal F$ and $\iota_{\mathcal {H}}$.
	\end{lemma}
	\begin{proof}
	We prove the result by induction on $m$. It is obvious for $m = 0,1$.
	Assume that the result holds for 
	any positive integer $m$ $(m\ge 2)$. 
	Now consider the case $m$.  Without loss of generality, suppose that $P=e_1,e_2,\cdots,e_l$ is the longest path.  
		Set 
	$$
	\mathcal H' =\mathcal H - {e_l}, \:
	\mathcal F' =\mathcal F \setminus\mathcal F_{e_l}.
	$$
		Then
	$\mathcal F'$ is a $k$-fold cover of $\mathcal H'$. There is a 
	bijection $f'$ between $\cal F'$ and the natural cover $\iota_{\mathcal H'}$ according to the induction hypothesis. For each $i \in \brk{k}$, let 
	$$
	f(\varphi_i^{(j)}(u))=
	\left\{
	\begin{array}{ll}
		f'(\varphi_i^{(j)}(u)),\qquad 
		&\mbox{ if }j\ne l;\\
		i, &\mbox{ otherwise. }
	\end{array} 
	\right.
	$$
	It is easy to verify that $f$ is a bijection between  its $\mathcal F$ and $\iota_{\mathcal H}$. The conclusion is clear for $m$. Therefore the result holds.
	\end{proof}

\begin{lemma}\label{lem:4-1}
	Suppose that a hypergraph $\mathcal H$ contains one and only one cycle
	$C$, and
	$e=\{v_i: i\in \brk{r}\}$ is an edge in $C$, where $r\ge 2$
	and $v_1$ and $v_2$ are the two vertices in $e$ which also appear in other edges of $C$. 
	Let $S_{k}^{(i_1,i_2,\cdots,i_r)}$ be the set of proper $k$-colorings of $\mathcal H-e$ that color $v_{l}$ by $i_l\in \brk{k}$ for all $l\in \brk{r}$, where $k\ge 2$. Then for any given $i_l\in\brk{k}$
	$$
	t_{1}=\frac{1}{k}(P(\mathcal H-e,k)-P(\mathcal H,k))
	$$
	and
	$$t_{2}=\frac{1}{k^{r-1}(k-1)}(k^{r-2}P(\mathcal H,k)+(1-k^{r-2})P(\mathcal H-e,k)),
	$$ 
	where  $t_1:=|S_{k}^{(i_1,i_2,\cdots,i_r)}|$ 
	when $i_1=i_2$, and $t_2:=|S_{k}^{(i_1,i_2,\cdots,i_r)}|$,
	when $i_1\ne i_2$.
\end{lemma}
\begin{proof}
Because $\mathcal H-e$ doesn't contain any cycle, 
$$
kt_{1}=P(\mathcal H-e,k)-P(\mathcal H,k)
$$
and
$$k(k^{r-2}-1)t_{1}+k^{r-1}(k-1)t_2=P(\mathcal H,k).$$
The conclusion is obtained. 
\end{proof}

\begin{lemma}\label{lem:4-2}
	Suppose that $\mathcal F$ is a 
	$k$-fold cover of a hypergraph $\mathcal H$,
	where $k\in \mathbb{N}$. 
	Suppose that  $\mathcal H$ contains one and only one cycle $C$, and 	
	$e=\{v_i: i\in \brk{r}\}$ is an edge in $C$,
	where $r\ge 2$. 
  If there is a 
	 bijection between the set of $\mathcal F^{'}$-colorings
	 and the set of its proper 
	 $k$-colorings of $\mathcal H-e$,
	  where $\mathcal F^{'}=\mathcal F\setminus\mathcal F_e$, then
	  there exists a $k$-fold cover $\mathcal F^{\ast} $ of $\mathcal H$ such that
	 \eqn{eq3}
	 {
	P_{DP}(\mathcal H, \mathcal F)
	&\geqslant &P_{DP}(\mathcal H, \mathcal F^*)
	\nonumber \\
	&=&\min\left \{ 	P(\mathcal H, k),
\dfrac{(k^{r-1}-1)P(\mathcal H-e,k)-
k^{r-2}P(\mathcal H,k)}
{k^{r-2}(k-1)}
		\right\} .
 }
\end{lemma}
\begin{proof}
	Suppose that 
	$\mathcal H_{1}$ is the connected component of $\mathcal H-e$ where $v_1\in V(\mathcal H_{1})$. Assume that $v_1$ and $v_2$ are the two vertices in $e$ which also appear in other edges of $C$. For any given $i_l\in \brk{k}$, let $S_{k}^{(i_1,i_2,\cdots,i_r)}$ be the set of proper $k$-colorings of $\mathcal H-e$ that color $v_l$ by $i_l$ for each $l\in \brk{r}$. 			
	Let $|S_{k}^{(i_1,i_2,i_3,\cdots,i_r)}|=t_{1}$ when $i_1=i_2$, and 
	and let $|S_{k}^{(i_1,i_2,\cdots,i_r)}|=t_{2}$
	when  $i_1\neq i_2$.

	Note that there is a bijection between its $\mathcal F'$-colorings and $k$-proper colorings of $\mathcal H-e$. If $v_{1}$ and $v_2$ have 
	the same color $i\in \brk{k}$
	in one coloring of $\mathcal H-e$, then 
	there are precisely $t_1$ $\mathcal F'$-colorings of $\mathcal H-e$ which contain $(v_1,i)$ and $(v_2,i)$ for $i\in \brk{k}$; otherwise 
	there are precisely $t_2$ $\mathcal F'$-colorings of $\mathcal H-e$ which contain $(v_1,i)$ and $(v_2,j)$ for $i\neq j,i,j\in \brk{k}$.
	Since $|\mathcal F_{e}|\leqslant k$, it follows that
	$$P_{DP}(\mathcal H, \mathcal F)\geqslant P(\mathcal H-e,k)-{\rm max}\{kt_1,kt_2\}.$$	
	By Lemma \ref{lem:4-1} $$t_{1}=\frac{1}{k}(P(\mathcal H-e,k)-P(\mathcal H,k))$$
	and
	$$t_{2}=\frac{1}{k^{r-1}(k-1)}(k^{r-2}P(\mathcal H,k)+(1-k^{r-2})P(\mathcal H-e,k)).$$

	A $k$-fold cover $\mathcal F^{*}$ of $\mathcal H$ is constructed as follows.
	If $t_1\geqslant t_2,$ 
	then for each $i\in \brk{k}$
	add $\varphi_i$ with $\varphi(v_j)=i$ for each $j\in [n]$ based on $\mathcal F'$.
	If
	$t_1<t_2$, 
	then add $\varphi_{i}:e\rightarrow \brk{k}$ for each $i\in \brk{k}$ such that $\varphi_{i}(v_{1})=i+1$ and $\varphi_{i}(v_{j})=i$ for $2\le j\le r$.
	
	It is easy to verify that $\mathcal F^{*}$ has the desired property. Therefore the result holds.
\end{proof}

\begin{theorem}\label{the:4-3}
	Suppose $\mathcal H$ is a unicycle linear $r$-uniform hypergraph with $m+p$ edges, and its unique cycle is of 
	length $p$ for $p\ge 3$, $r\ge 3$ and $m\ge 0$. 
	Then, 
	\begin{enumerate}
		\item[(1)]  If $p$ is odd, then 
		$$
		P_{DP}(\mathcal H,k)=P(\mathcal H,k).
		$$
		\item[(2)] 
		 If $p$ is even, then 
		$$P_{DP}(\mathcal H,k)=(k^{r-1}-1)^{m+p}+(-1)^{p+1}(k^{r-1}-1)^{m}.$$
	\end{enumerate} 
\end{theorem}

\begin{proof}
Without loss of generality,
 suppose that 
 $E(\mathcal H)=
 \{e_j: j\in\brk{m+p}\}$ and $C$ is the only cycle in $\mathcal H$, and 
 $e=\{v_i: i\in\brk{r}\}$ is an edge in $C$,
 where $v_1$ and $v_2$ are the two
 vertices in $e$ which also appear
 in other edges of $C$. 
  Since adding $\varphi_p^{(j)}$ to $\mathcal F$ can only make the number of $\mathcal F$-colorings of $\mathcal H$ smaller, suppose that 
   $\mathcal F=\{\varphi_i^{(j)}: i\in \brk{k}, j\in \brk{m}\}$ is 
  a $k$-fold cover of $\mathcal H$ with $k\ge 2$. 
   Then $\mathcal H-e_1$ is a hypertree.
	By Lemma \ref{lem:3-3} and Lemma \ref{lem:4-2},
	\begin{equation}	\label{eq:4-3}
		P_{DP}(\mathcal H, \mathcal F)\geqslant	
		P(\mathcal H-e_1,k)-{\rm max}\{A,B\}
	\end{equation}
	where $A=P(\mathcal H-e_1,k)-P(\mathcal H,k)$ and $ B=\dfrac{k^{r-1}P(\mathcal H,k)+(k-k^{r-1})P(\mathcal H-e_1,k)}{k^{r-1}(k-1)}.$
	
	By Lemmas \ref{lem:2-1} and \ref{lem:2-3}, 
	$$A=(k^{r-1}-1)^{m+p-1}+(-1)^{p+1}(k-1)(k^{r-1}-1)^m$$
	and
	$$B=(k^{r-1}-1)^{m+p-1}+(-1)^{p}(k^{r-1}-1)^m.$$
	
	If $p$ is odd, then $A> B$. By applying the inequality \eqref{eq:4-3}, 
	$$P_{DP}(\mathcal H,\mathcal F)\ge P(\mathcal H,k).$$
	By the inequality \eqref{eq:3-1},
	$P_{DP}(\mathcal H,k)= P(\mathcal H,k).$
	If $p$ is even, then $A< B$. By applying the inequality \eqref{eq:4-3}, 
	$$P_{DP}(\mathcal H,\mathcal F)\ge (k^{r-1}-1)^{m+p}+(-1)^{p+1}(k^{r-1}-1)^{m}.$$
	By Lemma \ref{lem:4-2}, there is $\mathcal F^*$ such that 
	$$P_{DP}(\mathcal H,\mathcal F^*)=(k^{r-1}-1)^{m+p}+(-1)^{p+1}(k^{r-1}-1)^{m}.$$
	Therefore, 
	$$P_{DP}(\mathcal H,k)=(k^{r-1}-1)^{m+p}+(-1)^{p+1}(k^{r-1}-1)^{m}.$$
\end{proof}

\begin{lemma} \label{lem:4-4}
	Let	$\mathcal T_m^{r}$ is a $r$-uniform hypertree with $m$ edges for $r\ge 2$. If $m\geqslant 0$, then
	$$P_{DP}(\mathcal T_{m}^{r},k)=P(\mathcal T_{m}^{r},k).$$ 
\end{lemma}
\begin{proof}
	It is obvious for $m=0,1$. Suppose that $\mathcal F$ is any $k$-fold cover of $\mathcal T_m^{r}$ and suppose that $e\in E({\cal T}_m^{r})$. 
	Lemmas \ref{lem:3-2}-\ref{lem:3-3}
	imply that
	
	$$P_{DP}(\mathcal T_{m}^{r},{\cal F})\geqslant P(\mathcal T_{m}^{r}-e,k)-\max\{P(\mathcal T_{m}^{r}-e,k)-P(\mathcal T_{m}^{r},k),\frac{P(\mathcal T_{m}^{r},k)}{k^{r-1}-1}\}.$$
	
	By Lemma \ref{lem:2-1},
	
	$$P(\mathcal T_{m}^{r}-e,k)-P(\mathcal T_{m}^{r},k)=k^{r-1}(k^{r-1}-1)^{m-1}-(k^{r-1}-1)^{m}=(k^{r-1}-1)^{m-1}$$
	and	
	$$\frac{P(\mathcal T_{m}^{r},k)}{k^{r-1}-1}=(k^{r-1}-1)^{m-1}.$$
	
	So $$P_{DP}(\mathcal T_{m}^{r},F)\geqslant P(\mathcal T_{m}^{r},k).$$
	Because $$P_{DP}(\mathcal T_{m}^{r},k)\leqslant P(\mathcal T_{m}^{r},k),$$
	the result is deduced. 
\end{proof}
The case of $r=2$ in the following result implies Corollary $4$ in \cite{KM21}.
\begin{theorem}
	Let	$\mathcal H$ is a connected $r$-uniform hypergraph with $n$ vertices and $m$ edges for $r\ge 2$. 
	Then, for any $k\in \N$,
	\equ{eq4}
	{
	P_{DP}(\mathcal H,k)= \dfrac{k^{n}(k^{r-1}-1)^{m}}
	{k^{(r-1)m}}
}
	if and only if $\mathcal H$ is a hypertree.
\end{theorem}
\begin{proof}
If the hypergraph is a hypertree, then
(\ref{eq4}) holds 
by Lemma \ref{lem:4-4}. Conversely, let suppose that $\mathcal H$ is a connected $r$-uniform hypergraph with $n$ vertices and $m$ edges such that
(\ref{eq4}) holds.

Let $A=\dfrac{k^{n}(k^{r-1}-1)^{m}}{k^{(r-1)m}}$.
If $n=(r-1)m$, then 
$A=(k^{r-1}-1)^{m}$. By 
Theorem \ref{the:4-3}, 
$$
P_{DP}(\mathcal H,k)<(k^{r-1}-1)^{m}.
$$ 
If $n\le (r-1)m-1$, then $A\le \dfrac{1}{k}(k^{r-1}-1)^{m}.$

	Because $\dfrac{1}{k}(k^{r-1}-1)^{m}$ isn't an integer for $k\ge 2$,
$$
P_{DP}(\mathcal H,k)<A.
$$
If $n=(r-1)m+1$, then $\mathcal H$ is a $r$-uniform hypertree.
By the given conditions on $\hyh$,
we have $n\le (r-1)m+1$. 
Therefore the conclusion holds. 
\end{proof}

\section{Further study}

In this section, some problems are proposed whose solutions could make progress on the similar problems in \cite{KM21,Th09}.

\begin{problem}{\rm For which connected hypergraphs $\mathcal H$ does $P_{DP} (\mathcal H, k) = P(\mathcal H, k)$ for every $k \ge 2$?}
\end{problem}
Kaul and Mudrock proved that the answer is affirmative for a tree, a unicycle graph of an odd cycle, chordal and $C_{2p+1}\oplus C_{2q+1}$ \cite{KM20}. 
In the article, we show that the answer is affirmative for a $r$-uniform hypertree and a connected unicycle linear $r$-uniform hypergraph with an odd cycle.

\begin{problem}{\rm
	For a hypergraph $\mathcal H$ with $P_{DP}(\mathcal H, k_0) = P(\mathcal H, k_0)$ for some $k_0\ge \chi(\mathcal H)$, is it true for all $m\ge m_0$ that
	$$P_{DP}(\mathcal H, k) = P(\mathcal H, k)?$$} 
\end{problem}

Kaul and Mudrock provided the affirmative answer of the following problem for a graph \cite{KM20}. 

\begin{problem}
	{\rm For every hypergraph $\mathcal H$, does there exist $p, N\in \mathbb N$ such that $P_{DP} (\mathcal H\vee K_p, k) =P(\mathcal H \vee K_p, k)$ whenever $k\ge N$?}
\end{problem}

\begin{problem}{\rm
Let $\chi_{DP}(\mathcal H)$ denote the smallest positive integer $k$ such that $\cal H$ admits a $\cal F$-coloring for every $k$-fold cover $\cal F$ of a hypergraph $\cal H$. Is it true that if 
$k > \chi(\mathcal H)$,
then
$P_{DP}(\mathcal H,k) > 1$?}
\end{problem}
Thomassen asked if there exists a graph $G$ and 
a $k > 2$ such that $P_l(G, k) = 1$. Clearly, one could make progress on this question by 
showing $P_{DP} (G, m) > 1$ for certain $G$ and $m\in \mathbb N$.

\begin{problem} {\rm For which connected hypergraphs $\mathcal H$ does $P_{DP} (\mathcal H, k) < P(\mathcal H, k)$ for some $k\in \mathbb{N}${\rm ?} Let $P_{DP} (\mathcal H, k) \approx P(\mathcal H, k)$ denote a hypergraph with $P_{DP} (\mathcal H, k) < P(\mathcal H, k)$ and $P_{DP} (\mathcal H, k) = P(\mathcal H, k)$ for enough $k$. Moreover, for which connected hypergraphs $\mathcal H$ does $P_{DP} (\mathcal H, k) \approx P(\mathcal H, k)$?}
\end{problem}
Readers are referred to several results of graphs for the problem above in \cite{DY22, KM21, MT21}.

\begin{problem}
{\rm	For any hypergraph $\mathcal H$ with $n$ vertices, if $P_{DP}(\mathcal H,k)<P(\mathcal H,k)$, is there a constant $\lambda$ such that  for sufficiently large $k$ 
	$$
	P(\mathcal H, k)-P_{DP} (\mathcal H, k) \le\lambda k^{n-2}?
	$$}
\end{problem}

\vskip 5mm
\noindent{\bf Acknowledgements }
\vskip 5mm

This work is supported by National Natural Science Foundation of
China under Grant
No. 12371340.

\end{document}